\documentclass[11pt]{amsart}
\usepackage{mathpazo}
\usepackage{latexsym}
\usepackage{amssymb}
\usepackage{amscd}
\usepackage{graphicx}

\newtheorem{thm}{Theorem}[section]
\newtheorem{prop}[thm]{Proposition}
\newtheorem{lem}[thm]{Lemma}

\theoremstyle{remark}

\theoremstyle{definition}

\newcommand{\bN}{\mathbb N}

\DeclareMathOperator{\Map}{Map}

\newcommand{\Q}{\mathbb{ Q}}

\makeatletter
\newcommand\norm{\bBigg@{0.8}}

\makeatother
\newcommand{\inparens}[2][flex]{\csname #1l\endcsname(#2%
  \csname #1r\endcsname)\mathclose{}}
\newcommand{\inangles}[2][flex]{\csname #1l\endcsname\langle#2%
  \csname #1r\endcsname\rangle\mathclose{}}

 \DeclareMathOperator{\cl}{cl}

\newcommand{\sv}[2][flex]{\csname #1l\endcsname\|#2%
  \csname #1r\endcsname\|} 
\newcommand{\lone}[2][flex]{\csname #1l\endcsname\|#2%
  \csname #1r\endcsname\|_1} 
\newcommand{\supn}[2][flex]{\csname #1l\endcsname\|#2%
  \csname #1r\endcsname\|_{\infty}} 

\title[On stability of non-domination under taking products]
      {On stability of non-domination\\ under taking products}
\author{D.~Kotschick}
\address{Mathematisches Institut, {\smaller LMU} M\"unchen,
Theresienstr.~39, 80333~M\"unchen, Germany}
\email{dieter@member.ams.org}
\author{C.~L\"oh}
\address{Fakult\"at f\"ur Mathematik, Universit\"at Regensburg, 93040 Regensburg, Germany}
\email{clara.loeh@mathematik.uni-regensburg.de}
\author{C.~Neofytidis}
\address{Department of Mathematical Sciences, {\smaller SUNY} Binghamton, Binghamton, NY 13902-6000, USA}
\email{chrisneo@math.binghamton.edu}
\date{July 6, 2015; \copyright{\ D.~Kotschick, C.~L\"oh and C.~Neofytidis 2015}}
\subjclass[2000]{57N65, 55M25}
\thanks{The first author is grateful to M.~Gromov for a long
  discussion several years ago, which inspired this note. The second
  author was supported by the CRC~1085 \emph{Higher Invariants}
  (Universit\"at Regensburg, funded by the DFG).  The third author is
  grateful to S.~Wang for useful conversations.}

\begin{document}

\begin{abstract}
We show that non-domination results for targets that are not dominated by products are stable under Cartesian products.
\end{abstract}

\maketitle

\section{Motivation}

If $M$ and $N$ are closed oriented manifolds of the same dimension, we say that \emph{$M$ dominates $N$}, and we write $M \geq N$, 
if there is a continuous map $f\colon M\longrightarrow N$ of non-zero degree. The existence of such a dominant map is a property 
of the homotopy types of $M$ and $N$, and it has been known since the pioneering work of Hopf~\cite{Hopf} that for such a map~$f$ 
the pullback $f^*$ is an injection of rational cohomology algebras, and that $f_*$ is virtually surjective on the fundamental group. 
However, the existence of an injective algebra homomorphism~$H^*(N;\Q) \longrightarrow H^*(M;\Q)$ and of a virtually surjective 
homomorphism~$\pi_1(M) \longrightarrow \pi_1(N)$ is usually far from sufficient for~$M \geq N$. 

Motivated by the work of Gromov~\cite{Gromov,gromovmetric} in particular, (non-)domination between manifolds has in recent
years been studied in several different contexts, using a variety of techniques from topology, geometry, and group theory; see for 
example~\cite{Gromov,CT,gromovmetric,connellfarb,KL} and the references given there.
An idea due to Thurston~\cite{MT} and Gromov~\cite{Gromov} 
is to study numerical invariants $I$ of manifolds that are monotone under maps of non-zero degree, 
so that $M\geq N$ implies $I(M)\geq I(N)$. Then, whenever one can compute or estimate $I$ and 
prove $I(M)<I(N)$ for some specific manifolds, one concludes that $M$ does not dominate $N$. The simplest example 
of such an invariant is the cuplength in rational cohomology, which is montone by the result of Hopf mentioned before.
A more subtle monotone invariant -- of geometric rather than algebraic origin -- is the simplicial volume $\sv {\ \cdot \ } $ defined by Gromov~\cite{Gromov}.
In general, monotone invariants are closely connected to functorial semi-norms on homology~\cite{gromovmetric,crowleyloeh,loeh}. 

According to Gromov, the simplicial volume has a major deficiency: its lack of multiplicativity. In fact, he proved in~\cite{Gromov} 
that the simplicial volume is approximately multiplicative for Cartesian products, and it is known that it is not strictly 
multiplicative~\cite{bucherkarlsson}. 
However, approximate multiplicativity is not good enough to obtain stable non-domination results. 
Indeed, suppose that $0<\sv M<\sv N$ for some specific $M$ and $N$. Then $M\ngeq N$, but it is unclear 
whether the $d$-fold product $M^{\times d}$ may dominate~$N^{\times d}$ for some~$d \geq 2$, or not. 
The approximate multiplicativity does not rule out the possibility that, as a function of the number of factors, 
the simplicial volume of direct products of~$M$ might grow faster than that of direct products of~$N$, 
so that the former eventually surpasses the latter. 

Invariants that are strictly multiplicative 
-- or strictly additive, like the cuplength -- do not have this deficiency: if $I(M)<I(N)$, then
$I(M^{\times d})<I(N^{\times d})$, so that $M^{\times d}\ngeq N^{\times d}$ for all $d\geq 1$. In this case the non-domination 
result $M\ngeq N$ is stable under Cartesian products.

Gromov~\cite{gromovmetric} suggested that many manifolds $N$ might have the property that they cannot be dominated 
by a non-trivial product $M=M_1\times M_2$. This conjecture has since been verified~\cite{KL}, and there are now lots of 
examples of manifolds that are known not to be dominated by products~\cite{KL,KL2,KN,Neo}. We will see here that
in general non-domination results for targets that cannot be dominated by products are stable under Cartesian products.
This is interesting in its own right, and also has geometric applications~\cite{Neo}.

\subsection*{Conventions}
Throughout this paper, the word manifold means a connected closed oriented non-empty topological manifold; 
we denote the rational fundamental class of a manifold~$M$ by~$[M]$. 
A product of manifolds is always a non-trivial product, so no factor is a point. 

\section{Results}

Our first result is that for targets that are not dominated by products, the loss of information in taking products 
discussed in the previous section does not occur.
\begin{thm}\label{t:main1}
Suppose $M$ and $N$ are $n$-manifolds, and that $N$ is not dominated by a product. 
Then for any $d \geq 2$ we have $M^{\times d}\geq N^{\times d}$ if and only if $M\geq N$.
\end{thm}

In a similar spirit, taking Cartesian products with arbitrary manifolds preserves non-domination for targets that
are not dominated by products.
\begin{thm}\label{t:main2}
Suppose $M$ and $N$ are $n$-manifolds, and that $N$ is not dominated by a product. Then for any manifold $W$, 
we have $M\times W\geq N\times W$ if and only if $M\geq N$.
\end{thm}
Note that $W$ may very well have trivial simplicial volume. Even if one deduces $M\ngeq N$ from 
$\sv M<\sv N$, this theorem shows that multiplying with $W$ preserves non-do\-mi\-na\-tion, 
while killing the simplicial volume if~$\sv W =0$.

Finally, controlling the dimensions of the factors in a product, we have the following:
\begin{thm}\label{t:finiteprod}
Let $N$ be an $n$-manifold that is not dominated by a product. Then there is no manifold 
$V$ for which the product $N\times V$ can be dominated by a product $P=X_1\times\ldots\times X_s$ that satisfies 
$\dim X_j<n$ for all $j \in \{1,\dots,s\}$.
\end{thm}

\section{Proofs}

The proofs of the above theorems all use the following lemma, which is a consequence of Thom's work~\cite{Thom}
on the Steenrod problem.

\begin{lem}\label{l:Thom}
Let $N$ be an $n$-manifold that is not dominated by a product. If 
$$f\colon M_1\times M_2\longrightarrow N$$
is a continuous map, then for all~$i \in \{1,\dots, n-1\}$ the map 
$$
f_*\colon H_i(M_1;\Q)\otimes H_{n-i}(M_2;\Q)\longrightarrow H_n(N;\Q)
$$
induced by the homological cross-product and~$f$ is the zero map.
\end{lem}
\begin{proof}
Because elements of $H_i(M_1;\Q)\otimes H_{n-i}(M_2;\Q)$ are finite linear combinations of decomposable elements, 
and $f_*$ is linear, it suffices to show $f_*(\alpha\otimes\beta )=0$ for all~$\alpha \in H_i(M_1;\Q)$ and all~$\beta \in H_{n-i}(M_2;\Q)$. 
Again by the linearity of $f_*$, there is no loss of generality in replacing $\alpha$ and $\beta$ by non-zero multiples. 
Thus we may assume that these are integral homology classes. By Thom's result~\cite{Thom}, after replacing the integral classes
$\alpha$ and $\beta$ by suitable non-zero multiples, there are continuous maps $g_j\colon X_j\longrightarrow M_j$ 
defined on manifolds $X_j$ of dimensions $i$ and $n-i$ respectively, such that $(g_1)_*[X_1]=\alpha$ and 
$(g_2)_*[X_2]=\beta$. It follows that 
$$
f_*(\alpha\otimes\beta ) = \bigl(f\circ (g_1\times g_2)\bigr)_*[X_1\times X_2] \ .
$$
This must vanish, because otherwise the map $f\circ (g_1\times g_2)\colon X_1\times X_2\longrightarrow N$
would have non-zero degree, contradicting the assumption on~$N$.
\end{proof}

Using Lemma~\ref{l:Thom}, we now prove the theorems stated in the previous section. 

\begin{proof}[Proof of Theorem~\ref{t:main1}]
If $M\geq N$, then clearly $M^{\times d}\geq N^{\times d}$ for all $d\geq 2$. Conversely, suppose that $g\colon M^{\times d}\longrightarrow N^{\times d}$
has non-zero degree for some~$d\geq 2$. We consider the composition $f=p_1\circ g$, where $p_1$ is the projection to the first factor. Then
$f_*$ is surjective in rational homology. Since we assumed that $N$ is not dominated by a product, Lemma~\ref{l:Thom}
tells us that, in degree $n$, the map $f_*$ vanishes on tensor products of homology vector spaces of non-zero degree. 
It follows that for at least one of the inclusions $i\colon M\longrightarrow M^{\times d}$ of a factor of $M^{\times d}$, the composition 
$f\circ i$ has non-zero degree, and thus~$M\geq N$.
\end{proof}

\begin{proof}[Proof of Theorem~\ref{t:main2}]
If $M\geq N$, then clearly $M\times W\geq N\times W$ for all manifolds~$W$. 
Conversely, suppose that $f\colon M\times W\longrightarrow N\times W$ has non-zero degree for some $W$. 
We consider the induced map $f_*$ on~$H_n(\;\cdot\;;\Q)$ in terms of the K\"unneth decompositions of the 
domain and of the target:
$$
f_*\colon H_n(M;\Q)\oplus \mathcal{M}_1 \oplus H_n(W;\Q)\longrightarrow H_n(N;\Q)\oplus \mathcal{M}_2 \oplus H_n(W;\Q) \ ,
$$
where $\mathcal{M}_i$ denotes the direct sum of tensor products of homology vector spaces in non-zero
degrees.

Since we assumed that $N$ is not dominated by a product, Lemma~\ref{l:Thom} tells us that $f_*(\mathcal{M}_1)$ is 
contained in $\mathcal{M}_2 \oplus H_n(W;\Q)$. If we assume for a contradiction that $M\ngeq N$, then the same is true for 
$f_*(H_n(M;\Q))$.

Because $f_*$ is surjective, we conclude that there is an $\alpha_0\in H_n(W;\Q)$ such that $f_*(\alpha_0) = [N]\neq 0$
holds in the quotient vector space 
\[ Q=H_n(N\times W;\Q)/f_*(H_n(M;\Q)\oplus\mathcal{M}_1) \ .
\] 
Note that $Q$ is of finite, non-zero, dimension.

Now we think of $\alpha_0$ as being in the target of $f_*$. By surjectivity of $f_*$, the class $\alpha_0$ is in its image,
so there exists an $\alpha_1\in H_n(W;\Q)$ satisfying $f_*(\alpha_1)=\alpha_0$ in $Q$ (though not necessarily in 
$H_n(N\times W;\Q)$).
We proceed inductively to find $\alpha_{i+1}\in H_n(W;\Q)$ with the property that $f_*(\alpha_{i+1})=\alpha_i$ in $ Q$.
The assumptions that $N$ is not dominated by a product, or by $M$, imply at every step that $\alpha_i$ 
does not vanish in the quotient $Q$.

Since $Q$ is finite-dimensional, there is a minimal~$k \in \bN$ such that $\alpha_0,\ldots,\alpha_k$ are linearly 
dependent in $Q$. There are then $\lambda_i\in\Q$ with $\lambda_k\neq 0$ such that
$$
\lambda_k\alpha_k+\ldots +\lambda_0\alpha_0=0 \in Q \ .
$$
We now take the left-hand-side of this equation, considered as an element of $H_n(W;\Q)\subset H_n(M\times W;\Q)$, and apply
$f_*$ to it to obtain
$$
\lambda_k\alpha_{k-1}+\ldots +\lambda_1\alpha_0+\lambda_0 [N]\in f_*(H_n(M;\Q)\oplus\mathcal{M}_1) \ .
$$
If $\lambda_0=0$, then this contradicts the minimality of $k$. If $\lambda_0\neq 0$, then we reach the 
conclusion that in $H_n(N\times W;\Q)$ the generator $[N]\in H_n(N;\Q)$ is a linear combination of 
$\lambda_k\alpha_{k-1}+\ldots +\lambda_1\alpha_0\in H_n(W;\Q)$ and of elements in 
$$
f_*\bigl(H_n(M;\Q)\oplus\mathcal{M}_1\bigr)\subset \mathcal{M}_2 \oplus H_n(W;\Q) \ .
$$
This contradicts the K\"unneth decomposition, and hence proves $M\geq N$.
\end{proof}

\begin{proof}[Proof of Theorem~\ref{t:finiteprod}]
Suppose 
$
g\colon X_1\times\ldots\times X_s\longrightarrow N\times V
$
is a continuous map, and consider the composition $f=p_1\circ g$. The assumptions that $N$ is not dominated 
by a product and that $\dim X_j<n$ for all $j$ imply, as in the proof of Lemma~\ref{l:Thom}, that $f_*$ is the zero map 
in degree $n$. Therefore, $g$ has degree zero.
\end{proof}

\section{Discussion}

\subsection{Applications of the cuplength}
It is not clear to what extent the assumption that $N$ is not dominated by a product is necessary in the above theorems. 
While it is crucial for our proofs, this could be an artefact of our method.
Indeed, there are cases of targets $N$ which are dominated by products, and still one can prove our results for them. 
We now do this for tori, using the cuplength. 

Recall that the cuplength of $M$, denoted $\cl (M)$, is the maximal number $k$ for which there are classes 
$\alpha_1,\ldots,\alpha_k\in H^*(M;\Q)$ of positive degrees with the property that 
$ \alpha_1\cup\ldots\cup\alpha_k\neq 0\in H^*(M;\Q)$. This is monotone under maps of non-zero degree 
by~\cite{Hopf}\footnote{Hopf did not use cohomology, but formulated the conclusion in terms of the {\it Umkehr} map 
on intersection rings.}. The compatibility of the K\"unneth decomposition with the cup product implies
\begin{equation}\label{eq:cl}
\cl (M\times W) = \cl (M) +\cl (W) \ .
\end{equation}

The following is easy and well known.
\begin{lem}\label{l:cl}
An $n$-manifold $M$ dominates $T^n$ if and only if there is an injective algebra homomorphism $H^*(T^n;\Q)\longrightarrow H^*(M;\Q)$,
equivalently, if $\cl (M)=n$.
\end{lem}
So this is a case where the algebraic necessary condition for domination derived from rational cohomology is also sufficient.

Lemma~\ref{l:cl} combined with~\eqref{eq:cl} tells us that Theorem~\ref{t:main2} holds for $N=T^n$. Furthermore, we have:
\begin{prop}
If $M_1$ and $M_2$ are manifolds of dimensions $m_1$ and $m_2$ respectively, then $M_1\times M_2\geq T^{m_1+m_2}$ if and 
only if $M_1\geq T^{m_1}$ and $M_2\geq T^{m_2}$.
\end{prop}
In particular, Theorem~\ref{t:main1} also holds for $N=T^n$.

\subsection{Infinite products}
Gromov has suggested that some non-domination results should extend to infinite products, following his perspective on
infinite products and related topics~\cite{Ber,Gro,BG}\cite[Section 5]{2000}.

By increasing the number $d$ of factors in $P^{\times d}$, one would naively end up with a countably infinite product $P^{\times \infty}$, without
any extra structure. A better way of looking at infinite products is probably to pick a (discrete, countable) group $\Gamma$, and to look
at the space~$P^{\Gamma}=\Map (\Gamma,P)$, equipped with the natural shift action of~$\Gamma$. 
Now in formulating what $P^{\Gamma}\ngeq N^{\Gamma}$
might mean, one should only consider $\Gamma$-equivariant continuous maps between these product spaces.

The main issue is of course that for maps between these
infinite-di\-men\-sional manifolds there is no naive, geometric, notion
of degree. Instead, one should make full use of equivariance and
define domination via surjectivity in a suitable homology theory,
perhaps without necessarily attempting to define a degree.

\bigskip

\bigskip

\end{document}